\documentclass[11pt]{article}
\usepackage{latexsym}
 \usepackage{amssymb}
 \usepackage{amsmath}
 \usepackage{amsthm}
 \usepackage{graphicx}

\newcommand{\cvl}{\stackrel{l}{\longrightarrow}}
\newcommand{\cvloi}{\stackrel{\mathbb{L}}{\longrightarrow}}

 \newcommand{ \un }{1\!\!1}
 \newcommand{ \p }{\mathbb{P} }
 \newcommand{ \tpa }{\mathbb{P}^{V} }
 \newcommand{ \pa }{\mathbb{P}^{V} }
  \newcommand{ \pao }{\mathbb{P}^{V}_{0} }

 \newcommand{ \Ea }{\mathbb{E}^{V}}

 \newcommand{ \tEa }{\mathbb{E}^{V}}
 \newcommand{ \tF }{\tilde{F}}
  \newcommand{ \bF }{\bar{F}}
 \newcommand{ \tS }{\tilde{\sum}}
 \newcommand{ \bS }{\bar{\sum}}
  \newcommand{ \tSn }{\tilde{\sum}_{(\log n)^{2- \epsilon}}}
  
  \newcommand{ \DL }{\lim_{n_1 \rightarrow +\infty}\lim_{n \rightarrow +\infty}}

\newcommand{ \dis }{ \mathcal{D} }

 \newcommand{ \R }{ \mathbb{R} }
 \newcommand{ \Z }{ \mathbb{Z} }
 \newcommand{\N}{ \mathbb{N} }
 
 \newcommand{ \V }{\textrm{Var} }

 \newcommand{ \Ct }{ \mathcal{C} }
  \newcommand{ \Bt }{ \mathcal{B} }
 
 \newcommand{ \f }{ \mathcal{F} }
 \newcommand{ \e}{ \mathcal{E} }

\newcommand{ \ct }{ \textrm{const} }

\newcommand{\Ht}{ \mathcal{H}}
\newcommand{\tpi}{\pi}
\newcommand{\ttpi}{\tilde{\pi}}

\newcommand{ \Rf}{ \mathcal{R}}
\newcommand{ \Rt}{ (\Rf_{m_n}^j)^{-1} }
\newcommand{ \Cm}{ \mathcal{C}_{m_n}^j }
\newcommand{ \Ce}{ \mathcal{C} }

\newcommand{ \A}{ \mathcal{A} }
 \newcommand{ \B}{ \mathcal{B} }

 \newcommand{ \lo }{ \mathcal{L} }

\newtheorem{The}{{\bf Theorem}}[section]
\theoremstyle{definition}

\theoremstyle{plain}
 \newtheorem{Lem}[The]{Lemma}
 \newtheorem{Cor}[The]{\bf Corollary}
 \newtheorem{Pro}[The]{\bf Proposition}
 \theoremstyle{definition}
\newtheorem{Rem}[The]{{\bf Remark}}

 \newenvironment{Pre}{\noindent \textbf{Proof.} \\ }{$\
 \blacksquare$}

\makeatletter

\@addtoreset{equation}{section} \makeatother

\title{The local time of a random walk on growing hypercubes
\\ \vspace{1cm}
 \large{Pierre Andreoletti} \footnote{Laboratoire MAPMO - C.N.R.S. UMR 6628 - F\'ed\'eration Denis-Poisson, Universit\'e d'Orl\'eans, 
(Orl\'eans France). \newline \vspace{0.1cm}  $\quad$  MSC 2000 60G50; 60J55. \newline \vspace{0.5cm} \textit{Key words and phrases :  random environment, reversible markov chain, Dirichlet method, recurrent regime, local time} } \textrm{ }    }

\begin{document}
\maketitle

\begin{abstract} 
We study a random walk in a random environment (RWRE) on $\Z^d$, $1 \leq d < +\infty$. The main assumption is that conditionned on the environment the random walk is reversible. Moreover we construct our environment in such a way the walk can't be trapped on a single point like in some particular RWRE but in some specific d-1 surfaces. These surfaces are basic surfaces with deterministic geometry. We prove that the local time in the neighborhood of these surfaces is driven by a function of the (random) reversible measure. As an application we get the limit law of the local time as a process on these surfaces. 
\end{abstract}

\section{Introduction  and results}

Multidimensional RWRE  have been studied a lot in the past thirty years in many different directions, so many directions that we can not be exhaustive here, we recommand  \cite{Zeitouni} for a survey on some of those directions. \\
One way to construct a RWRE is to start with a reversible markov chain, in our case it will be a reversible nearest neighborhood random walk. We also assume that the reversible measure is a function of a certain potential $V$, which is a function from $\Z^d$ to $\R$.  Now if  this potential also called environment is random we get a RWRE. \\
Under additional assumptions on $V$, R. Durrett \cite{Durett} shows that unlike the simple random walk in $\Z^d$ for $d \geq 3$ , these RWRE can be recurrent (for almost all environment) and sub-diffusive. He also points out that the walk localize itself but due to the dimension $d>1$ the point of localization is not easy to characterize contrary to the one-dimension case of Ya. G. Sinai \cite{sinai}. 
Also he gets the joint distributions of the logarithm of the occupation times of the random walk (also called local time), before it reaches a certain level set of $V$. This is this last aspect we are interested in, in fact it is the link between the local time and the reversible measure, which is random for a RWRE, that we study in this paper. \\
The random environment $(V)$ we built here is different from the ones of R. Durrett: we build a potential that creates $d-1$ dimensional valleys. Indeed the idea that a particle can be trapped on a surface like a sphere or an hypercube, for example, by opposite random fields  is pretty physical. We also want the particle to be able to escape from a surface to reach another one where stronger fields act. Finally, in addition to be random, we do not want the fields to be uniform on a given surface, but that the number of values taken on this surface is finite. The definition of $V$ we will consider is the following: first $V$ is a function of two elementary processes. The first one, denoted $(S_k,k \in \Z)$, is a sum of i.i.d random variables. The second one, denoted $(\delta_x,x \in \Z^d)$, is a sequence of i.i.d random variable indexed by the vertices of $\Z^d$. This two processes are assumed to be independent. Then the potential $V$ on a point $x$ of $\Z^d$, is a  sum of random variables ($S$) with random number of terms ($\bar{x}+\delta_x$) where $\bar{.}$ is a norm on $\Z^d$. The way we construct this environment make it, at a given distance of the origin, strongly dependant from a site to the other.\\ Naturally we add hypothesis on the random environment, that is to the increments of $S$, and the distribution of the $\delta's$. With this hypothesis detailed below we get an almost surely recurrent random walk for almost all environment. It appears also that the walk is trapped on a set of points, whose shape depends on the norm $\bar{.}$. In fact, like in the one dimensional case, the random environment creates valleys, and as we restrict our analysis to the infinite norm on $\Z^d$ the shape of the valleys are hypercubes. 

An important fact is that these hypercubes, where the walk is trapped, are not level sets for the potential neither for the reversible measure denoted $\pi$, however 
they can be partitioned in level sets for $\pi$. A part of this work is devoted to the study of the local time of the walk on these level sets.

In the same way N. Gantert, Y. Peres and Z. Shi \cite{GanPerShi} (see also \cite{AndDiel} for the one-dimension continuous case) give the limit behavior of the local time in the neighborhood of the coordinate of the depth of a specific  valley we do it here but for a set of points. One of the main difference is the fact that we have to deal with the local time on a set of points and therefore several values of the reversible measure are involved.

The ingredient for the study of the local time of a random walk is mainly the computation of the moments of the excursions of the walk. Excursions are the amount of time the walk, starting from a point of the lattice (or a set of points), spends before it returns to this point (or this set of points). This moments can be written as an explicit function of  probabilities for the walk to touch a point before another.
It appears that for the (nearest neighborhood) 1-dimensionnal RWRE these probabilities  can be expressed explicitly from $V$. In dimension larger than $1$ it is not the case in general so other technics can be used, like Dirchlet methods for example. However this method leads in general to inequality instead of equality and therefore weaker results in term of precision and  convergence.
Our purpose is to show that we can be as precise as in the one dimensional case for non-trivial examples of multidimensional RWRE.

The paper is organized as follows, Sections 1.1 and 1.2 are devoted to the definitions of the basic processes together with a general description of the random environment. In  Section 1.3 we state our main results, and in Section 1.4 we give some examples. In Section 2 the reader will find the proof of the results, it is devided into three sub-sections, the first one concerns the random environment, the second the quenched result and the last one the annealed result. Also we add an appendix making this paper quite self-containt.

\subsection{Definitions and hypothesis}

 \textit{The random environment,} is a random sum of random variables:
first let us denote $(S_k,k \in \Z)$ the following sum 
\begin{eqnarray} 
 && S_{k}:=\left\{ \begin{array}{ll} \sum_{1\leq i \leq k} \eta_i, & \textrm{ if }\ k =1,2, \cdots \\
  \sum_{k+1 \leq i \leq 0} - \eta_i , &   \textrm{ if }\  k=-1,-2,\cdots \end{array} \right.  \nonumber \\
&& S_{0}:=0, \nonumber
 \end{eqnarray}
where $(\eta_i,i\in \Z)$ is a i.i.d. sequence of random variables, we denote $(\Omega_1,\f_1,P_1)$ the probability space associated to $(S_k,k \in \Z)$. Notice that $S_.$ is the typical potential of a one-dimensional random walk in random environment. \\ Also let $(\delta_x,x \in \Z^d)$ another sequence  of i.i.d. random variables  with integer values independent of the sequence $(\eta_i,i\in \Z)$, and $(\Omega_2,\f_2,P_2) $ is the associated probability space.
Let $d \in \N^*$, the potential $V:\Z^d \rightarrow \R$ also called random environment is defined by, for all $x \in \Z^d$ 
 \begin{eqnarray*}
V(x):= \left\{\begin{array}{l}  S_{\bar{x}+ \delta_x}, \textrm{  if } \bar{x}+ \delta_x \geq 0,  \\
 0 \textrm{ instead. }
 \end{array} \right. 
\end{eqnarray*}
where $\bar{x}$ is the infinite norm in $\Z^d$ $(\bar{x}=\max_{1 \leq j \leq d}\left\{|x_j|\right\})$. Also we denote $(\Omega_e,\f_e,P_e)$ the probability space induced by $(V(x),x \in \Z^d)$.

\noindent \textit{The random walk,} let us fix $V$, define a reversible Markov chain like an electical network following for example \cite{DoySne} or \cite{RusPer}. First we assign to each edge (x,y) of $\Z^d$ 
 a conductance denoted $\pi(x,y)$ defined as follows
 \begin{eqnarray*}
 \pi(x,y):=\exp(-1/2V(x)-1/2V(y))\un_{\{|x-y|=1\}}.
 \end{eqnarray*}
where $\un$ is the indicator function and $|x-y|=1$ means that $x$ and $y$ differ only by one coordinate $(|x-y|=1 \iff x=y\pm e_i, 1 \leq i \leq d)$. Also define the capacitances $\pi$, for all $x \in \Z^d$
 \begin{eqnarray*}
 \pi(x):=\sum_{z \in \Z} \pi(x,z) = \exp(-1/2V(x)) \sum_{z, |x-z|=1}\exp(-1/2V(z)).
 \end{eqnarray*}
Then, let $(X_n,n\in \N)$ a markov chain with state space $\Z^d$ and probability of transition $p$ given by, for all $x \in \Z^d$ and $y \in \Z^d$,
\begin{eqnarray*}
p(x,y):= \frac{\pi(x,y)}{\pi(x)}= \frac{\exp(-1/2V(y))}{ \sum_{z, |x-z|=1}\exp(-1/2V(z))}\un_{\{|x-y|=1\}}.
\end{eqnarray*}
By construction, $(X_n,n\in \N)$ is reversible and the reversible measure is given by
$\pi(.)$, notice also that $p(x,y)\neq p(y,x)$ so the walk is not symmetric. We denote $(\Omega_3,\f_3,\pa)$ the probability measure associated to this random walk and we denote $\pa_z(\cdot):=\pa(\cdot|X_0=z)$, for all $z \in \Z^d$. Finally the whole process, also denoted $X_.$, is defined on a probability space denoted $(\Omega, \f, \p)$. In particular for all $A_1\in \f_1$, $A_2\in \f_2$ and $A_3 \in \f_3$, $\p(A_1\times \A_2 \times A_3)=\int_{A_3}d\p^{V(w_1,w_2)}(w_3)\int_{A_2,A_1}dP_2(w_2)dP_1(w_1)$.\\
\textit{Hypothesis on the random environment,}
to study this walk we add hypothesis on $V$, we choose the following
\begin{eqnarray}
& & E_1[\eta_0]=0, \label{hyp1} \\
& & \V[\eta_0]>0,\label{hyp2}  \\
& & |\eta_0|  \textrm{ is bounded }P_1.a.s, \label{hyp3} \\
& & \delta_{0_{d}}  \textrm{ is bounded }P_2.a.s, \label{hyp4}
\end{eqnarray}
and $0_{d}$ is the origin of $\Z^d$. Under those hypothesis we can prove, see \cite{Pierre8}, that the walk is recurrent for almost all environnment, it is sub-diffusive, not localized in the neighborhood of a point of $\Z^d$  but on an hypercube of $\Z^d$ centred in $0_{d}$, and at a random distance of the origin. \\ 
\textit{Local time on hypercubes}, here we are interested in the asymptotic behavior of the local time $\lo$ of $X$:  let $x \in Z^d$ and $n \in \N$
$$\lo(x,n)=\sum_{k=1}^n\un_{X_k=x},$$
it is the time spent by the walk at the point $x$ within the intervall of time $[\![ 1,n]\!]$,   
moreover for any $U \subset \Z^d$ define $\lo(U,n)=\sum_{x\in U} \lo(x,n)$, the time spent by the walk in the sunbet $U$. More especially we are interested in the asymptotic behavior of the local time on the sets $\Ct_k$ with $k \in \N^*$ define as follows
\begin{eqnarray*}
\Bt_{k}&:=& \{y \in \Z^d, \bar{y} \leq k \}, \nonumber \\
\Ct_{k}&:=& \Bt_{k} \smallsetminus \Bt_{k-1}. 
\end{eqnarray*}
Notice that for $d=1$, $\Ct_{k}$ is reduced to 2 points, so we treat that case separately in Section \ref{dim1}.
As an application of our results, we will get the convergence in law of the supremum of the local time $\sup_{k \in \Z} \lo(C_k,n)$.
The results we present depend on some more definition on the random environment, we define it in the following section.

\subsection{Values of $\pi$ on $\Ct_k$ and potential conditionned to stay positive} 

In this section we assume $d \geq 2$.

\subsubsection{Possible values of $V$ on $\Ct_k,\ k \in \N$}

Let $\e^0$ the space state of $\delta_{0_d}$ and $|\e^0|$ its cardinal, from the definitions we give above it is clear that $V$ can take, at most, $|\e^0|$ values on a $\Ct_k$. Figure \ref{fig51} is a graphical representation of the values of $V$ on $\Ct_k$ if d=2 and if we assume, as an example, that $\delta$ can only take 3 values: -1, 0 and 1. A large segment outside the square means that at this point of the square $\delta_.=1$, a large segment inside the main square means that at this point of the square $\delta_.=-1$, no segment means that at this point of the square $\delta_.=0$.

\begin{figure}[hh]
\begin{center}
\input{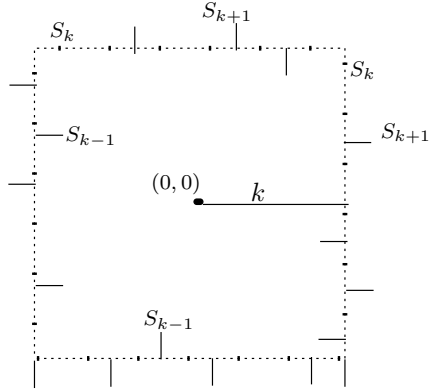} 
\caption{$d=2$, different values of $V$ on $C_k$} \label{fig51}
\end{center}
\end{figure} 
Unfortunatly, we will see that it is not the number of different values that takes the potential that is of importance for our results but the one of the reversible measure:

\subsubsection{Values of the reversible measure on $\Ct_k$, $k \in \N$}

Let $k \in \Z$, define $\mathcal{V}_{k,d}^{\pi}:=\{\textrm{disctinct } \pi(x), x \in \Ct_k \}$. 
Thanks to Hypothesis \ref{hyp4}, we know that $|\mathcal{V}_{k,d}^{\pi}|$ is bounded  and constant for $k$ large enough, for simplicity we denote this constant
\begin{eqnarray*}
n_d:=|\mathcal{V}_{k,d}^{\pi}|.
\end{eqnarray*}
Let, $1 \leq j \leq n_d$, we denote $(\pi^j_k,1 \leq j \leq n_d)$ the elements of $\mathcal{V}_{k,d}^{\pi}$ and 
\begin{eqnarray*}
\Ct_k^j:=\{ z \in \Ct_k, \pi(z)=\pi^j_k \},
\end{eqnarray*}
a subset of $\Ct_k$ which is a \textit{level set} for $\pi$. Of course $\Ct_k=\bigcup_{j=1}^{n_d}\Ct_k^j$, moreover we can give a simple expression for the $\pi^j_k$, let $x \in \Z^d$ we recall that 
\begin{eqnarray*}
\pi(x)=\exp(-1/2V(x)) \sum_{i=-d, i \neq 0}^d\exp(-1/2V(x+\textrm{sign}(i)e_{|i|}))
\end{eqnarray*}
where $(e_i,1 \leq i \leq d)$ is the canonical base of $\Z^d$. Let us define $\bar{\Ct}_k$ the sub-set of $\Ct_k$ without the edge i.e. $\bar{\Ct}_k=\{z \in \Z^d, \bar{z}=k, \#\{z_i, 1 \leq i \leq d :  z_i=k \} = 1 \}$, also denote $"\cdot"$ the standard Euclidean scalar product.  If we assume that $x \in \bar{\Ct}_k$ with for example $x \cdot e_1=k$, then $\overline{x\pm e_1}=k\pm1$, whereas $\overline{x\pm \textrm{sign}(i)e_{|i|}}=k$  for the other $i \neq 1$. Therefore following this example we get for all $x\in \bar{\Ct}_k$ with $x\cdot e_1=k$,
\begin{eqnarray*}
\pi(x) &=& \exp(-1/2S_{k+\delta_x}) \times \\ 
& & \left(\exp(-1/2S_{k+1+\delta_{x+e_1}})+\exp(-1/2S_{k-1+\delta_{x-e_1}})+ \sum_{i=-d, i \neq 0,1}^d\exp(-1/2S_{k+\delta_{x+\textrm{sign}(i)e_{|i|}}})\right).
\end{eqnarray*}
When $x \in \Ct_k \smallsetminus \bar{\Ct}_k$, we can get similar expressions. The conclusion of this is that, for all $k$ and $j$, there exists a sequence of integers $(a_0^j,a_1^j,\cdots,a_{2d}^j)$ such that 
\begin{eqnarray}
\pi^j_k= \exp(-1/2S_{k+a_0^j}) \sum_{l=1}^{2d}\exp(-1/2 S_{k+a_l^j}). \label{valpi}
\end{eqnarray}
Note that $a_0^j$ takes its values in $ \e^0$ and the other $a_i^j$ take their values on one of the following three sets, $\e^0$, $\e^{+}:=\{l+1, l \in \e^0\}$ and $\e^-:=\{l-1, l \in \e^0\}$. We also define $\e:=\{l-1,l,l+1, l \in \e^0\}$.
\\

\noindent To resume we can say that the $d-1$ hypercubes of $\Z^d$ $\Ct_k$ can be partitioned on level sets for $\pi$. Moreover, due to the $\delta$'s, the $\Ct_k^j$ have a random geometry, and the values of $\pi$ on these level sets are given by \ref{valpi}.

\subsubsection{Limit law of $S$ conditioned to stay positive \label{parSp}}
As we can see in the above expression of $\pi$, the reversible measure on the hypercubes involves the knowledge of random sums of random variables. Now looking at the expression of the $\pi^j_k$, we see that for fixed $\delta_.$, the reversible measure only depends on a sum of i.i.d. random variables $(S_.)$ which pretty looks like the one dimensional case (see for example \cite{GanPerShi}). \\
Let us recall basic random variable used in the one-dimensional case:
\begin{eqnarray*}
M_n &:=& \inf\{k>0: S_k-\min_{0\leq i\leq k}S_i \geq \log n+(\log n)^{1/2}\} \\
m_n&:=&  \inf\{k>0: S_k=\min_{0\leq i\leq M_n}S_i \} 
\end{eqnarray*}
These two random variables correspond to a the coordinate of the top ($M_n$) and the  bottom ($m_n$) of a one dimensionnal localization valley.
\noindent In what follows we also need to introduce $(\bar{S}_k,k)$ a sequence of random variables with the same distrbution as $S$ and conditioned to stay non-negative for $x>0$ and strictly positive for $x<0$. \cite{Golosov} and \cite{Bertoin} has shown independently that under \ref{hyp1} $\bar{S}$ is well defined and also that $\sum_{x=-\infty}^{+\infty}\exp(-\bar{S}_x)<+\infty,\ a.\ s.$ 
We are ready to define $\bar{\pi}^j_k$ for all $1 \leq j \leq n_d$ and $k \in \Z$
\begin{eqnarray*}
\bar{\pi}^j_k:= {\exp(-1/2\bar{S}_{k+a_0^j}) \sum_{l=1}^{2d}\exp(-1/2 \bar{S}_{k+a_l^j})} 
.
\end{eqnarray*}


\subsection{Results - Local time on hypercubes}

In this section we also assume that $d \geq 2$ and we start with the annealed results

\subsubsection{Annealed results}

\noindent Our first result deals with the limit distribution of the supremum of the local time taken over all the hypercubes $(\Ct_l,l \in \N)$:
\begin{The} \label{th1.1} Assume hypothesis \ref{hyp1}-\ref{hyp4} are satisfied, there exists a sequence of non-negative real $(p_j, 1 \leq j \leq n_d)$ with $\sum_{j=1}^{n_d} p_j=1$ and a sequence of vectors $(a_0^j,a_1^j,\cdots,a_{2d}^j, 1 \leq j \leq n_d)$ belonging to $\e^{\otimes n_d}$ such that
\begin{eqnarray*}
\ \frac{\sup_{l\in \N}\lo(\Ct_l,n)}{n} \cvloi \sup_{i \in \Z} \bar{\Pi}(i),  
\end{eqnarray*}
where for all $i \in \Z$
\begin{eqnarray}
\bar{\Pi}(i)=\frac{\sum_{j=1}^{n_d}p_j\bar{\pi}^{j}_i}{{\sum_{l=-\infty}^{+ \infty} \sum_{j=1}^{n_d}p_j\bar{\pi}^{j}_l }}, \label{1.14}
\end{eqnarray}
and $\cvloi$ is the convergence in law under $\p$. 
\end{The}

This result shows that the behavior of the local time depends on the different values taken by the invariant (reversible) measure on the hyper cubes $\Ct_k$. Of course the $\pi$'s depends itself of the values taken by the $\delta$ and $\bar{S}$. 
If we enter in the details, we see that like in the one dimensionnal case the reversible measure and the porcess $\bar{S}$  are the important ingredients, and that the $\delta$'s act like a noise. The new things is first that as we are looking  on a set of points, several values of the reversible measure appear (the numerator of $\pi_i^j$). The second new things is that the values of the reversible measure do not have the same importance, indeed the distribution of the $\delta$'s can give more importance to some of them and neglect other. This information is contained in the $p_j$, in particular, for this general case, we allow them to be nul. \\
 In the next section, we take some examples for the distribution of the $\delta$'s and show that all the sequences (the $(p_j,j)$ and the $(a_0^j,a_1^j,\cdots,a_{2d}^j,j)$) that appears here become explicit.

The above result is an easy consequence of the following result, in that one we look at the local time as a process on specific (random) hypercubes. 

\begin{The} \label{th1.3} Assume hypothesis \ref{hyp1}-\ref{hyp4} are satisfied, there exists a sequence of non-negative real $(p_j, 1 \leq j \leq n_d)$ with $\sum_{j=1}^{n_d} p_j=1$ and a sequence of vectors $(a_0^j,a_1^j,\cdots,a_{2d}^j, 1 \leq j \leq n_d)$ belonging to $\e^{\otimes n_d}$ such that
\begin{eqnarray*}
\left( \frac{\lo(\Ce_{m_n+i},n)}{n}, {i \in \Z} \right) \cvloi \left(\bar{\Pi}(i) , i \in \Z\right) 
\end{eqnarray*}
where the $\bar{\Pi}$, is given by \ref{1.14}.
\end{The}

\subsubsection{Quenched result}

The importance of the quenched results, a part from the fact that they are the key results to get the annealed ones, is that it contains the information that leads from the local time to the reversible measure.

\begin{Pro} Assume hypothesis \ref{hyp1}-\ref{hyp4} are satisfied. Let $\delta>0$, there exists a sequence of vectors $(a_0^j,a_1^j,\cdots,a_{2d}^j, 1 \leq j \leq n_d)$ belonging to $\e^{\otimes n_d}$, also for all $n$ there exists a subset $A_{n} \in \f_e$ with $\lim_{n \rightarrow + \infty} P_e(A_n) =1$ such that
\begin{eqnarray*}
\lim_{n \rightarrow + \infty}\inf_{V \in A_n}\p^V_{0}\left(\bigcap_{l \in \{-k, \cdots,k\}}\left\{\left|\frac{\lo(\Ce_{m_n+l},n)}{n}-\sum_{j=1}^{n_d}\Rf_{m_n+l}^j\right|\leq \delta \sum_{j=1}^{n_d}\Rf_{m_n+l}^j \right\} \right)=1,
\end{eqnarray*}
where 
\begin{eqnarray*}
\Rf^j_{p}= \frac{|\Ct_p^j|\pi_{p}^j}{ \sum_{x \in \B_{M_n}} \pi(x)}.
\end{eqnarray*}
\end{Pro}
This proposition tells us that looking at the local time on the hypercubes $(\Ce_{m_n+l},l)$ is like looking at the weighted distinct values that can take the reversible measure in this set. The weight are given by the amount of points taken by the reversible measure on the  hypercube $(|\Ct_{m_n+l}^j|)$ and normalized by the "total mass" $(\sum_{x \in \B_{M_n}} \pi(x))$.   
\noindent Notice that $\Rf^j_{p}$ can be re-written in the following way 
\begin{eqnarray*}
\Rf^j_{p}= \frac{|\Ct_p^j|\pi_{p}^j}{\sum_{i=0}^{M_n} \sum_{j=1}^{n_d}|\Ct_i^j|\pi_{i}^j}.
\end{eqnarray*}
 In fact we can be more precise, and replace the denominator of $\Rf^j_{p}$ by $ \sum_{i=-(\log n)^{2-\epsilon}}^{(\log n)^{2-\epsilon}} \sum_{j=1}^{n_d} |\Ce^j_{m_n+i}| \tpi^j_{m_n+i}$, and then understand the real $p_j$ of the annealed result. Indeed, thanks to the fact that the $\delta_.$'s are i.i.d., $ |\Ce^j_{m_n+i}|/|\Ct_{m_n}| \rightarrow p_j, P_e.a.s.$ so the $p_j$ are just the proportion of a given reversible measure on an hypercube.
By changing the $A_n$ for an $A_n'$ with the same property as $A_n$ we can replace $\Rf^j_{p}$ by the following:
\begin{eqnarray}
\tilde{\Rf}^j_{p}= \frac{p_j\pi_{p}^j}{\sum_{i=-(\log n)^{2-\epsilon}}^{(\log n)^{2-\epsilon}} \sum_{j=1}^{n_d}p_j\pi_{i}^j}.
\end{eqnarray}
Notice also that we are looking to the walk on moving hypercubes: they are randomly  growing because their distance from the origin is $m_n$. Finally, the constants $(p_j, 1 \leq j \leq n_d)$ and vectors $(a_0^j,a_1^j,\cdots,a_{2d}^j, 1 \leq j \leq n_d)$ are preserved from an hypercube to an other, they purely depend on the random environment, especially the $\delta_.$'s.  \\
To the question: what can we learn about one weighted reversible measure from the local time ? The above proposition is not very useful, because it mixes different values of the reversible measure. However the above proposition is a corollary of  the following statement
\begin{eqnarray*}
\lim_{n \rightarrow + \infty}\inf_{V \in A_n}\p^V_{0}\left(\bigcap_{l \in \{-k, \cdots,k\}}\bigcap_{j=1}^{n_d} \left\{ \left|\frac{\lo(\Ce_{m_n+l}^j,n)}{n}-\Rf_{m_n+l}^j\right|\leq \delta \Rf_{m_n+l}^j \right\} \right)=1,
\end{eqnarray*}
also we can compare the strength of the different weighted $\pi$'s from the local time by the following:
\begin{eqnarray*}
\lim_{n \rightarrow + \infty}\inf_{V \in A_n}\p^V_{0}\left( \bigcap_{j=1}^{n_d} \bigcap_{i \neq j} \left\{ \left|\frac{\lo(\Ce_{m_n}^j,n)}{\lo(\Ce_{m_n}^i,n)}-\frac{p_j\pi_{m_n}^j}{p_i\pi_{m_n}^i}\right|\leq \delta  \right\} \right)=1.
\end{eqnarray*}





\subsection{Examples}

In this section we give a short analysis of three cases. The first one is a trivial case: we assume that the $\delta_{0_d}\equiv 0$, Theorem \ref{th1.1} get of course simpler, and can be seen as a one-dimensional case. The second example is the simplest we can build with random and non trivial $\delta$'s, we assume that they are Bernoulli trials with parameter $p$, in this case we easily get explicit formula for the $p_i$'s and the $\pi^j_i$'s. Our third case is the one-dimensional case, we just point out that it is not exactly Sinai's walk (\cite{sinai}, \cite{GanPerShi}), even if it pretty looks like it.

\subsubsection{The trivial case, $\delta \equiv 0$, $d \geq 2$}
In this case the reversible measure on $\Ct_k$ can only takes $d$ different values and by definition $\Ct_k$ are level sets for $V$, however $\Ct_k$ is not a level set for the reversible measure $\pi$. Indeed there is some edges exception: for example when d=2 $\pi$ takes two different values on $\Ct_k$: $\forall x \in \bar{\Ct}_k,\ \pi(x)=(2\exp(-S_k)+\exp(-S_{k+1})+\exp(-S_{k-1}))$, but $\forall x \in \Ct_k\setminus  \bar{\Ct}_k, \ \pi(x)=(2\exp(-S_k)+2\exp(-S_{k+1}))$. In general it is easy to check that $ \forall x \in \bar{\Ct}_k,\ \pi(x)=\pi(k,0_{d-1})$.
We can represent the level sets in the 3-dimensionnal case as in figure \ref{fig5}: surfaces delimited by the black lines represent the level sets for $V$ (on the left) and the level sets of $\pi$ (on the right). 
\begin{figure}[h]
\begin{center}
\input{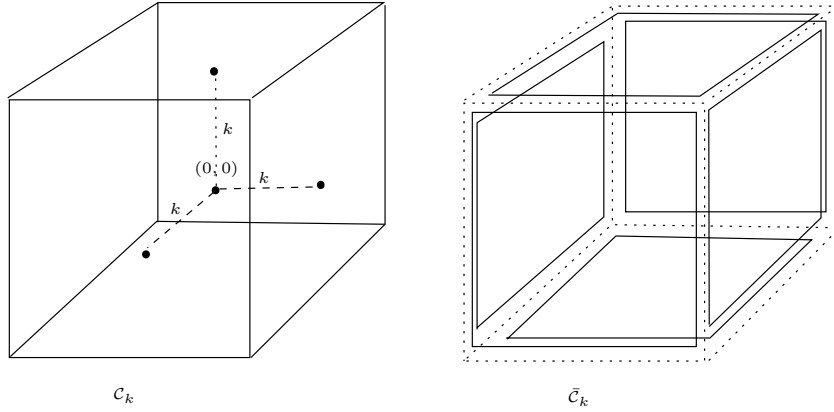} 
\caption{$d=3$, on the left level sets for $V$, on the right for $\pi$} \label{fig5}.
\end{center}
\end{figure} 

\begin{Rem} \label{1.1}
\noindent A few words on the size of the sets $\Ct_k$ and $\bar{\Ct}_k$, by size we mean the number of points in this set, we have $|\Ct_k|=k^{d}-(k-2)^{d}$, also we easily check that $|\Ct_k \smallsetminus \bar{\Ct}_k|=k^d-(k-2)^{d}-2d(k-2)^{d-1}$ therefore $|\bar{\Ct}_k|=2d(k-2)^{d-1}$. Finally for large $k$ we have $|\Ct_k| \thickapprox |\bar{\Ct}_k| \thickapprox  k^{d-1}$, and $|\Ct_k \smallsetminus \bar{\Ct}_k|  \thickapprox k^{d-2} $ . \\
\end{Rem}

Thanks to the preceding Remark there is only one value of the reversible measure  associated with a strictly positive probability $p_1$, which is of course equal to one. For example Theorem \ref{th1.1} is given by

\begin{Cor}  Assume hypothesis \ref{hyp1}-\ref{hyp3} are satisfied and $\delta_{0_d}=0$, then
\begin{eqnarray*}
\ \frac{\sup_{k\in \N}\lo(\Ct_k,n)}{n} \cvloi \sup_{l \in \Z} \bar{\Pi}(l), 
\end{eqnarray*}
where for all $l \in \Z$
\begin{eqnarray*}
\bar{\Pi}(l)=\frac{\exp(-1/2\bar{S}_l)\left[(2d-2)\exp(-1/2\bar{S}_l)+\exp(-1/2\bar{S}_{l-1})+\exp(-1/2\bar{S}_{l+1})\right]}{ \sum_{l=-\infty}^{+\infty}  \exp(-1/2\bar{S}_l)\left[(2d-2)\exp(-1/2\bar{S}_l)+\exp(-1/2\bar{S}_{l-1})+\exp(-1/2\bar{S}_{l+1})\right]}.
\end{eqnarray*}
\end{Cor}
When we look at the expression of $\bar{\Pi}(l)$, we understand that this case can be seen as a one dimension case: see Section \ref{dim1}.

\subsubsection{The Bernoulli case $\delta \thicksim B(p)$, $d \geq 2$}

This is the simplest case we can define with non trivial random $\delta$'s. In this paragraph we show that we can get a quite simple expression for $\bar{\Pi}$. Theorem \ref{th1.1} becomes:

\begin{Cor} \label{cor1.1} Assume hypothesis \ref{hyp1}-\ref{hyp3} are satisfied and that $\delta_{0_d}$ is a Bernoulli with parameter $p$, then
\begin{eqnarray*}
\ \frac{\sup_{k\in \N}\lo(\Ct_k,n)}{n} \cvloi \sup_{i \in \Z} \bar{\Pi}(i), 
\end{eqnarray*}
where for all $i \in \Z$
\begin{eqnarray*}
 \bar{\Pi}(i)&=&\frac{\Gamma_i}{{\sum_{l=-\infty}^{+ \infty} \Gamma_l}}, \\
\Gamma_i& =& \sum_{i_0=0}^1  \sum_{i_1=-1}^0 \sum_{i_2=1}^2 \sum_{k=0}^{2d-2} p^{B}(i_0,i_1,i_2,k) \pi^{B}(i,i_0,i_1,i_2,k), \\
 p^{B}(i_0,i_1,i_2,k)&=&p^{i_0+i_1+i_2}(1-p)^{3-i_0-i_1-i_2}\binom{2d-2}{k} p^{k}(1-p)^{2d-2-k}, \\
 \pi^B(i,i_0,i_1,i_2,k)&=&e^{-\bar{S}_{i+i_0}}\left(e^{-\bar{S}_{i+i_1}}+e^{-\bar{S}_{i+i_2}}+ ke^{-\bar{S}_{i+1}}+(2d-2-k)e^{-\bar{S}_{i}}\right).
\end{eqnarray*}
\end{Cor}

Like for the trivial case, we are less interested in the value of $n_d$ (the number of distinct reversible measures we can obtained on a surface $\Ct_k$ for some $k$ large enough), than in $\tilde{n}_d$ which is the number of distinct reversible measure associated with a non-nul $p_i$. Observing the expression for $\Gamma_i$ we easily get that $\tilde{n}_d=8(2d-2)$. Notice that the $p_i$'s are replaced by the  $p^{B}(i_0,i_1,i_2,k)$, and of course we can check that $ \sum_{i_0=0}^1  \sum_{i_1=-1}^0 \sum_{i_2=1}^2 \sum_{k=0}^{2d-2} p^{B}(i_0,i_1,i_2,k)=1.$
If we assume for the moment that the general case is proved, we can get the above corollary, just by showing the
\begin{Lem} Assume $\delta_{0_d}$ is a Bernoulli with parameter $p$, then the only values, associated with non-nul  $p_i$'s, that can take the reversible measure $\pi$ on a $\Ct_k$ are given by the $\pi^B$, moreover the associated number $p_.$ are given by the $p^{B}$.
\end{Lem}

\begin{proof} The proof is pretty easy, however it may be used for more general case so we give some details.
First thanks to Remark \ref{1.1}, we only have to consider the number of distinct ${\pi}$ on $\bar{\Ct}_k$. 
Let us define a vector $v$ such that its coordinate $(v_1,v_2, \cdots,v_{2d+1}) \in \{0,1\} \times \{1,2\}\times \{-1,0\} \times \{0,1\}^{2d-2}$. Now assume that $x\in \bar{\Ct}_k$ with $x\cdot e_1=k$, we recall that 
 \begin{eqnarray*}
 & & \pi(x) \\ 
 &=& e^{(-1/2S_{\bar{x}+ \delta_x})} \left( e^{(-1/2S_{\bar{x}+1+ \delta_{x+e_1}})} +e^{(-1/2S_{\bar{x}-1+ \delta_{x-e_1}})}+\sum_{i=-d,i \notin\{ 0,1-1\}}^d e^{(-1/2S_{\bar{x}+ \delta_{x+sign(i)e_i}})}\right)
 \end{eqnarray*}
So the vector $v$ represents one configuration for $\pi(x)$, example $(v=(0,2,-1,1,0_{2d-2}))$ means that $\delta_x=0,\delta_{x+e_1}=1,\delta_{x-e_1}=0, \cdots$. 
To get the result we can use, for example, a tree representation (see also Figure \ref{fig5b}): the root is given by the vector $v_0=(0\ or\ 1,-1,0_{2d-2})$, the first coordinate of the vector $v$ having no importance. For the first generation we fix or increment the second and third coordinate (in Figure \ref{fig5b} we underlined the fixed coordinates). Once all the second and third coordinate are fixed, each of these vectors give $2^d$ distinct vectors with fixed coordinate. Each of this last vectors leads to $2d-2$ distinct reversible measures and we are done. $\tilde{n}_d$ is given by 2 $(0\ or\ 1)$ times $4$ times the number of this vectors. We can read the expression of $\pi^B$ and the probability $p^{B}$ from the tree.  
\begin{figure}[h]
\input{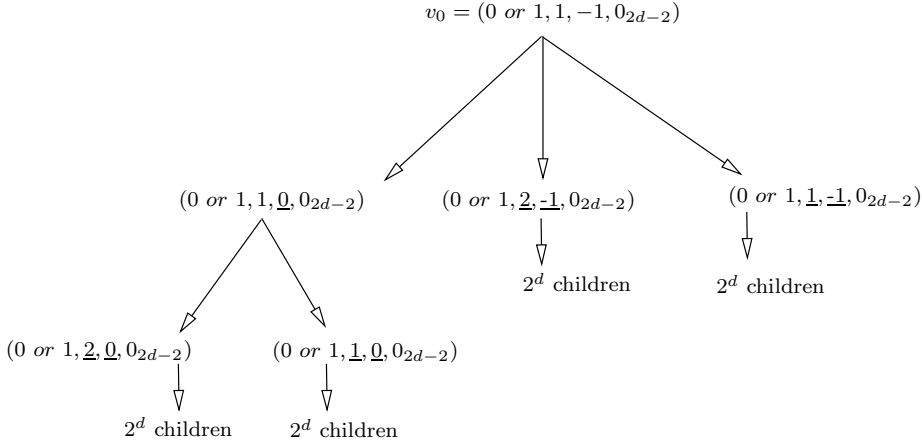} 
\caption{Tree representation for the Bernoulli case.} \label{fig5b}
\end{figure}
\end{proof}



\subsubsection{About dimension one \label{dim1} (d=1)}

With the definition of the RWRE we give we do not get exactly Sinai's random walk, even if we assume $\delta_{0_d} \equiv 0$. First because the probability of transition are not the same and also because in our case the random environment is almost symmetric with respect to the origin. Almost because if $\delta_{0_d} \equiv 0$ the random variables $\delta_{0_d}$ differs at each site. In \cite{Pierre8}, it is shown that we do not have necessarily a localization in one point but in one or two points each with a positive probability. Here,  by definition, we have $\Ct_k=\{-k,k\}$ so unlike the d-dimensional with $d>1$, the size of $\Ct_k$ does not grow with $k$. To compare,
with the result in \cite{GanPerShi} we assume that the walk is reflected at the origin so $\Ct_k=\{k\}$ first we recall
\begin{The} (\cite{GanPerShi}) Assume that $\lo_S$ is the local time of Sinai's walk (reflected at $0$), then 
\begin{eqnarray*}
\ \frac{\sup_{l\in \N}\lo_S(l,n)}{n} \cvloi \sup_{i \in \Z}\frac{e^{-\bar{S}_{i-1}}+e^{-\bar{S}_{i}}}{\sum_{k=-\infty}^{+ \infty}e^{-\bar{S}_{k-1}}+e^{-\bar{S}_{k}}}.
\end{eqnarray*}
\end{The}
In our case, we can say something like
\begin{Cor} Assume $d=1$, hypothesis \ref{hyp1}-\ref{hyp4} are satisfied then 
\begin{eqnarray*}
 \frac{\sup_{l\in \N} \lo(l,n)}{n} &\cvloi&  \sup_{i \in \Z} \bar{\Pi}(i), \textrm{ where} \\
\bar{\Pi}(i) &=&\frac{e^{-1/2\bar{S}_{i+ \delta_i}} \left(e^{-1/2 \bar{S}_{i-1+\delta_{i-1}}}+e^{-1/2 \bar{S}_{i+1+\delta_{i+1}}}\right)}{{\sum_{k=-\infty}^{+ \infty} e^{-1/2\bar{S}_{k+\delta_{k}}} \left(e^{-1/2 \bar{S}_{k-1+\delta_{k-1}}}+e^{-1/2 \bar{S}_{k+1+\delta_{k+1}}}\right) }}. \\
\end{eqnarray*}
\end{Cor}

To compare with the result of \cite{GanPerShi}, just take $\delta_{0_d}=0$ and get
\begin{eqnarray*}
\bar{\Pi}(i)=\frac{e^{-1/2\bar{S}_{i}} \left(e^{-1/2 \bar{S}_{i-1}}+e^{-1/2 \bar{S}_{i+1}}\right)}{{\sum_{k=-\infty}^{+ \infty} e^{-1/2\bar{S}_{k}} \left(e^{-1/2 \bar{S}_{k-1}}+e^{-1/2 \bar{S}_{k+1}}\right) }}. \\
\end{eqnarray*}

Notice that the one dimensional case do not behave like the multi-dimensional one, in particular we see that the $\delta$ appears explicitly when $d=1$. This just comes from the fact that the $\Ct_k$ are reduced to one point. \\
We will not give any details of the proof of this result, indeed we can get it just by mimic the proof in \cite{GanPerShi}.

\subsection{Ending remarks}

In \cite{GanPerShi} the result in law (equivalent of our Theorem) leads quite easily to the almost sure behavior of the $\limsup$ of the supremum of the local time. Here, except in some simple cases including the preceeding examples, some more work are needed that the reason why we treat this aspect in a seperate paper (see \cite{Pierre9}).
\\
The assumption \ref{hyp4}, about the $\delta$'s, allows us to keep the structure of valleys, things are different if the support of the $\delta$'s  depends on $x\in \Z^d$. For example for $\delta$'s not i.d. and with a support depending on $\bar{x}$ it is easy to construct an environment without valleys. However for a small support comparing to a typical fluctuation of the walk we believe that we can get a quite similar behavior. \\
Also we choose the infinite norm for the definition of the random environment which naturally leads to our hypercubes, however for other norms one can get similar results.





\noindent \\ The rest of the paper is devoted to the proof of the results.

\section{Proof of the results}

In all this section we assume that $d\geq 2$. 
We start with some notions and results for the random potential $V$,

\subsection{Some properties of the random potential}

We need some basic definitions and elementary properties on the random environment, we start with definition well known for the one dimensionnal case: 
\\
\noindent
Let $n\in \N$, we define 
\begin{eqnarray*}
\Delta_n = \left\{ \max_{ 0 \leq k \leq l \leq m_n } \max_{x \in \Ct_k}\max_{y \in \Ct_l}\left( V(y)- V(x) \right) \right\} \vee \left\{ \max_{ m_n \leq k \leq l \leq M_n } \max_{x \in \Ct_k}\max_{y \in \Ct_l}\left( V(x)- V(y) \right) \right\}
\end{eqnarray*}
In words $\Delta_n$ is the maximum between the highest height of potential the walk can encountered, when going from $0_d$ to $\Ct_{m_n}$ and the one when going from $\Ct_{M_n}$ to $\Ct_{m_n}$.  

We continue with the


\begin{Rem} \label{rem2.3}
thanks to hypothesis \ref{hyp3} and \ref{hyp4} ( the $\delta's$ and the increments of $S$ are bounded) we have, for all $m>0$  
\begin{eqnarray*}
S_{m} - \ct \leq S_{m+\delta_.} \leq S_{m}+\ct.
\end{eqnarray*}
\end{Rem}

\noindent We are now ready to state a lemma for the random environment $V$:
\begin{Lem} \label{lemenv}  Assume \ref{hyp1}-\ref{hyp4} are satisfied, let $n \in \N$, and  $A_{n}\in \f_{e}$ be such that $\forall \omega \in A_n$, $S(\omega)$ and $V(\omega)$ satisfy the following 3 properties, \\
1. For all $x \in \Ct_{m_n}$ and all $y \in \Ct_{M_n}$ $\left| V(y)-V(x) -\log n- (\log n)^{1/2} \right| \leq  \ct  $, \\
2. Let $\epsilon>0$, $ (\log n)^{2- \epsilon} \leq m_n \leq (\log n)^{2+ \epsilon}$, $(\log n)^{2-\epsilon} \leq M_n \leq (\log n)^{2+\epsilon}$,  $M_n/m_n \leq (\log n)^{\epsilon}$,\\
3. $\Delta_n Ê\leq \log n(1-\epsilon_n)$, with $\epsilon_n=\frac{(\log \log n)^2}{\log n}$, \\
then $\lim_{n \rightarrow + \infty}P_e(A_n)=1$. 
\end{Lem}

The proof is very basic, thanks to the above Remark we move from $V$ to $S$ easily, we are left to prove a one-dimensionnal result that can be found in the literature (see for example \cite{Pierre2} page 1389). Notice that Lemma \ref{lemenv} takes only into account the basic properties of the environment we need, some more advance statements are needed especially to obtain the annealed result. We present and show them Section 2.3. As usual, we write $V \in A_n$ to say that we are looking at an event $\omega \in \f_e$ such that $V(\omega)$ satisfies the three conditions.

\subsection{Proof of Proposition \ref{th1.3} (quenched results)}

We start with some basic statements, one saying that the walk is trapped in a specific valley, the other one that the walk reaches a level set $\Cm$ on a negligible amount of time comparing to $n$. In all this section we assume that $V \in A_n$.



\begin{Lem} Let $1 \leq j \leq n_d$, we have
\begin{eqnarray} 
\lim_{n \rightarrow \infty} \pa_0\left(T_{\Cm} \leq n^{1-\epsilon_n/2} \right)=1, \label{2.2}
\end{eqnarray}
for all $z \in \Ct_{m_n}$
\begin{eqnarray} 
 \pa_z(T_{\bar{\Bt}_{M_n}} > n ) \geq 1-(\log n)^{2(d-1)}n^{-(\log n)^{-1/2}}. \label{2.3}  
\end{eqnarray}
\end{Lem}
\begin{Pre}
The proof is based on the estimation of certain probability involving hitting times, we recall what we need at the end of the Appendix.  \\
\textit{Proof of \ref{2.3}} we use the classical way in that situation: we show that the walk prefers to return $n$ times to $\Ct_{m_n}$ instead of reaching $\bar{\Bt}_{M_n}$ the complementary of $\Bt_{M_n}$ in $\Z^d$ (notice also that $T_{\bar{\Bt}_{M_n}}=T_{\Ct_{M_n+1}}$). Let us define the following  stopping times, let $p\geq 1$, $A \subset \Z^d$,
\begin{eqnarray*}
&& T_{A,p } \equiv \left\{\begin{array}{l}  \inf\{k>T_{A,p-1},\ X_k\in A \}, \\
 + \infty \textrm{, if such }  k \textrm{ does not exist.}
 \end{array} \right. \\
& & T_{A} \equiv T_{A,1};\ T_{A,0}=0. 
\end{eqnarray*}
For simplicity, we also denote $T_{A}^+$ the first return to $A$ starting from $A$.
Let $y \in \Ct_{m_n}$, the strong Markov property yields:
\begin{eqnarray*}
\pa_y(T_{\bar{\Bt}_{M_n}}>T_{\Ct_{m_n},n})  \geq  \left( \min _{z \in \Ct_{m_n}} \pa_z\left(T_{\bar{\Bt}_{M_n}}>T^+_{\Ct_{m_n}}\right) \right)^n \geq  \left( 1- \max _{z \in \Ct_{m_n}} \pa_z\left(T_{\bar{\Bt}_{M_n}}>T^+_{z}\right) \right)^n. 
\end{eqnarray*}
 \ref{4.17} together with item 1. of Lemma \ref{lemenv} yields $\pa_z\left(T^+_{z}<T_{\bar{\Ct}_{M_n}}\right) \leq \ct |\B_{M_n}|\exp(-\log n-(\log n)^{1/2})$. So we get \ref{2.3} thanks to the facts $|\Bt_{m_n}|=(2m_n+1)^d$, $T_{\Ct_{m_n},n} \geq n$, and item 2 of Lemma \ref{lemenv}.
\\
\textit{Proof of \ref{2.2}} we do that in two steps, \textit{first} we show that  
\begin{eqnarray} 
\lim_{n \rightarrow \infty} \pa_0(T_{\Ct_{m_n}} \leq n^{1-\epsilon_n/2} )=1, \label{2.4b}
\end{eqnarray}
We have $T_{\Ct_{m_n}}=\lo(\Bt_{m_n},T_{\Ct_{m_n}})$, 
Markov inequality implies $\pa(\lo(\Bt_{m_n},T_{\Ct_{m_n}}) \geq n^{1-\epsilon_n} ) \leq n^{-1+\epsilon_n}$ $ \sum_{x \in \Bt_{m_n}} \Ea(\lo(x,T_{\Ct_{m_n}})) $ moreover $\Ea(\lo(x,T_{\Ct_{m_n}}))=\pa(T_x<T_{\Ct_{m_n}})/\pa_x(T_x^+>T_{\Ct_{m_n}})\leq 1/\pa_x(T_x^+>T_{\Ct_{m_n}})$. 
\noindent Thanks to \ref{4.19}, for all $x\in \Bt_{m_n}$, $\pa_x(T_x^+>T_{\Ct_{m_n}}) \geq  \ct\exp(-\Delta_n) |\Bt_{m_n}|^{-1} $, finally  thanks to item 2 and 3 of Lemma \ref{lemenv} we get 
\begin{eqnarray*}
\pao(\lo(\Bt_{m_n},T_{\Ct_{m_n}}) \geq n^{1-\epsilon_n} ) \leq \ct \frac{(\log n)^{2d(2+\epsilon)}}{\exp((\log \log n)^2 /2)},
\end{eqnarray*}
with $\epsilon>0$ and we \ref{2.4b} comes. To get \ref{2.2} from \ref{2.4b} we only have to show that for all $\epsilon>0$ for any $j: 1\leq j \leq n_d$ and $z\in \Ct_{m_n}$   
\begin{eqnarray} 
\lim_{n \rightarrow + \infty} \pa_z(T_{\Ct_{m_n}^j} \leq n^{1-\epsilon} )= 1. \label{2.4bb}
\end{eqnarray}
There is nothing to do if $z\in \Cm$, if not we get it with a similar argument used for \ref{2.3}.  \\
\end{Pre} \\



Our first key result is the following d-dimensional equivalent of Theorems 3.8 and 3.14 in \cite{Pierre2} or equation 2.10 in \cite{GanPerShi},  

\begin{Pro} Let $\epsilon>0$. For all $1 \leq j \leq n_d$,
\begin{eqnarray}
\lim_{n \rightarrow + \infty} \sum_{y_0\in \Cm}\p^V_{y_0}\left(\left|\frac{\lo(\Cm,n)}{n}-\Rf_{m_n}^j\right|\geq \epsilon \Rf_{m_n}^j \right)=0, \label{2.6}
\end{eqnarray}
where 
\begin{eqnarray*}
\Rf^j_{p}= \frac{|\Ct_p^j|\pi_{p}^j}{ \sum_{x \in \B_{M_n}} \pi(x)},
\end{eqnarray*}
moreover
\begin{eqnarray}
\lim_{n \rightarrow + \infty} \p^V_{0}\left(\left|\frac{\lo(\Ct_{m_n},n)}{n}-\sum_{j=1}^{n_d}\Rf_{m_n}^j\right|\geq \epsilon \sum_{j=1}^{n_d}\Rf_{m_n}^j \right)=0, \label{2.6b}
\end{eqnarray}
and if $k \in \N^*$
\begin{eqnarray}
\lim_{n \rightarrow + \infty}\p^V_{0}\left(\bigcup_{l \in \{-k, \cdots,k\}}\left\{ \left|\frac{\lo(\Ce_{m_n+l},n)}{n}-\sum_{j=1}^{n_d}\Rf_{m_n+l}^j\right|\geq \epsilon  \sum_{j=1}^{n_d}\Rf_{m_n+l}^j\right\}\right)=0. \label{2.8}
\end{eqnarray}
\end{Pro}

\begin{Pre}
\textit{Proof of \ref{2.6}} 
Let us denote  $\A_n=\left\{\left|\frac{\lo(\Cm,n)}{n}-\Rf_{m_n}^j\right|\geq \epsilon \Rf_{m_n}^j\right\}$ and $\dis_n=\{T_{\bar{\Bt}_{M_n}} > n\}$. It is easy to check that $\A_n\cap \dis_n \subset \A_n^+\cup \A_n^-$, where $\A_n^{\pm}=\{\pm \lo(\B_{M_n},T_{\Cm,n\Rf_{m_n}^j(1\mp\epsilon) })\geq \pm n   \}$ (assuming for simplicity that $n\Rf_{m_n}^j (1\pm \epsilon)$ are integers). Note that thanks to \ref{2.3} 
\begin{eqnarray*}
\lim_{n \rightarrow + \infty} \sum_{y_0\in \Cm}  \tpa_{y_0}(\A_n) \leq \lim_{n \rightarrow + \infty} \left( \sum_{y_0\in \Cm}\tpa_{y_0}(\A_n^+)+  \sum_{y_0\in \Cm} \tpa_{y_0}(\A_n^-)\right).
\end{eqnarray*} 
Let us give an estimate of $\tpa_{y_0}( \A_n^+)$. Denoting $n_{\epsilon}=n\Rf_{m_n}^j(1-\epsilon)$ and applying Markov inequality:
\begin{eqnarray}
\tpa_{y_0}( \A_n^+) & \equiv  & \tpa_{y_0} \left( \lo(\B_{M_n},T_{\Cm,n_{\epsilon}})-n_{\epsilon}\Rt  \geq \epsilon n\right), \nonumber \\
& \leq & \tEa_{y_0} \left(\left( \lo(\B_{M_n},T_{\Cm,n_{\epsilon}})-n_{\epsilon}\Rt \right)^2 \right)\left(\epsilon n\right)^{-2}.  \label{2.10}
\end{eqnarray}
Note that, as $\Cm$ is not a singleton, $\lo(\B_{M_n},T_{\Cm,j}) $  can not be written as a sum of i.i.d random variables and $\tEa_{y_0} \left(\left( \lo(\B_{M_n},T_{\Cm,n_{\epsilon}})-n_{\epsilon}\Rt \right)^2 \right)$ is not the variance of $\lo(\B_{M_n},T_{\Cm,n_{\epsilon}})$. For any $w \in \Cm$ let us denote $\underbar{E}_{w}=\tEa_{w}\left[(\lo({\B_{M_n}},T^+_{\Cm})-\Rt)\right]$, from Lemma \ref{lem3.3}, we have: 
\begin{eqnarray*}
& & \sum_{y_0\in \Cm}\tEa_{y_0} \left(\left( \lo(\B_{M_n},T_{\Cm,n_{\epsilon}})-n_{\epsilon}\Rt \right)^2 \right)   \\ &=&   n_{\epsilon} \sum_{y_0 \in \Cm} \tEa_{y_0} \left[\left(\lo(\B_{M_n},T^+_{\Cm})-\Rt\right)^2\right] \\ 
 &+& 2  \sum_{i=1}^{n^{\epsilon}-1} (n_{\epsilon}-i)  \sum_{v_i\in \Cm} \sum_{v \in \Cm} \underbar{E}_{v_i}  \underbar{E}_v \p_v^V(X_{T^+_{\Cm,i}}=v_i).
 \end{eqnarray*} 
For the \textit{first sum} we apply Cauchy-Schwarz inequality and get 
\begin{eqnarray}
\tEa_{y_0} \left[\left(\lo(\B_{M_n},T^+_{\Cm})-\Rt\right)^2\right]\leq |\B_{M_n}| \sum_{z  \in \B_{M_n}}\tEa_{y_0} \left[ \left(\lo(z,T^+_{\Cm})-\frac{\pi(z)}{\pi_{m_n}^j|\Cm|}\right)^2\right], \nonumber
\end{eqnarray} 
then with the help of Lemma \ref{Lemvar}  of the Appendix, and Remark \ref{rem2.3}, 
\begin{eqnarray}
\sum_{y_0\in \Cm}\tEa_{y_0} \left[\left(\lo(\B_{M_n},T^+_{\Cm})-\Rt\right)^2\right] & \leq & |\B_{M_n}| \sum_{z  \in \B_{M_n}} \frac{\tpi(z)}{\pi_{m_n}^{j}}\frac{2}{\tpa_z(T_z>T^+_{\Cm})} \nonumber \\
& \leq & \ct |\B_{M_n}|\sum_{z  \in \B_{M_n}} \frac{1}{\tpa_z(T_z>T^+_{\Cm})}.  \nonumber
\end{eqnarray} 
Finally thanks to Lemma \ref{Lem4.6} equations \ref{4.19} and \ref{4.19b} together with Lemma \ref{lemenv} statement 3
\begin{eqnarray}
\sum_{y_0\in \Cm}\tEa_{y_0} \left[\left(\lo(\B_{M_n},T^+_{\Cm})-\Rt\right)^2\right] &  \leq & \ct  |\B_{M_n}|^3n^{1-\epsilon_n},  \label{2.12}
\end{eqnarray} 
and we recall that $\epsilon_n=((\log \log n)^2) / \log n$.
\noindent \\ For the \textit{second sum}, Lemma \ref{A.14} shows that for $n$ large enough and $i \geq |\Cm|^{1+\epsilon_1}$  with $\epsilon_1>0$
\begin{eqnarray*} 
\left|\tpa_v\left(X_{T^+_{\Cm,i}}=v_i \right)-\frac{1}{ |\Cm|} \right| \leq \exp(- |\Cm|^{\epsilon_1}), 
\end{eqnarray*} 
 so asympoticaly when $n$ increases $\tpa_v\left(X_{T^+_{\Cm,i}}=v_i \right)$  does not depend on $v$, moreover thanks to Lemma \ref{ap2} $\sum_{v \in \Cm}\underbar{E}_v=0 $, therefore we get for $n$ large enouh \\
\begin{eqnarray*}
 & & \sum_{i=1}^{n_{\epsilon}-1} (n_{\epsilon}-i)  \sum_{v_i\in \Cm} \sum_{v \in \Cm} \underbar{E}_{v_i}  \underbar{E}_v \tpa_v(X_{T^+_{\Cm,i}}=v_i) \\
 & = &  \sum_{i=1}^{ |\Cm|^{1+\epsilon_1}} (n_{\epsilon}-i)  \sum_{v_i\in \Cm} \sum_{v \in \Cm} \underbar{E}_{v_i}  \underbar{E}_v \tpa_v(X_{T^+_{\Cm,i}}=v_i).
 \end{eqnarray*}
Finally for $n$ large enough
\begin{eqnarray*}
 & & \left|\sum_{i=1}^{n_{\epsilon}-1} (n_{\epsilon}-i)  \sum_{v_i\in \Cm} \sum_{v \in \Cm} \underbar{E}_{v_i}  \underbar{E}_v \tpa_v(X_{T^+_{\Cm,i}}=v_i)\right| \\ & \leq & n_{\epsilon} |\Cm|^{1+\epsilon_1}\left(\sum_{y_0\in \Cm}  \tEa_{y_0}\left[\lo({\B_{M_n}},T^+_{\Cm})\right]\right)^2,
 \end{eqnarray*}
and by using Lemma \ref{ap1} and Remark \ref{rem2.3}  we obtain
\begin{eqnarray}
 \left|\sum_{i=1}^{n_{\epsilon}-1} (n_{\epsilon}-i)  \sum_{v_i\in \Cm} \sum_{v \in \Cm} \underbar{E}_{v_i}  \underbar{E}_v \tpa_v(X_{T^+_{\Cm,i}}=v_i)\right| \leq \textrm{const} n_{\epsilon} |\Cm|^{1+\epsilon_1}|\B_{M_n}|^2. \label{2.14}
\end{eqnarray} 
To finish we collect \ref{2.10}, \ref{2.12}  and \ref{2.14}, we get 
\begin{eqnarray*}
\sum_{y_0 \in \Ct_{m_n}^j} \tpa_{y_0}( \A_n^+) \leq   \frac{ \ct \Rf_{m_n}^j |\B_{M_n}|^3 }{\epsilon^2 n^{\epsilon_n}}+ \frac{\textrm{const}\Rf_{m_n}^j |\Cm|^{1+\epsilon_1} |\B_{M_n}|^2}{ \epsilon^2 n}.
\end{eqnarray*}
 By definition $\Rf_{m_n}^j \leq 1$, so we get $\lim_{n \rightarrow + \infty}\sum_{y_0 \in \Ct_{m_n}^j}\tpa_{y_0}( \A_n^+)=0$ by applying Remark \ref{1.1} and  statement 2 of Lemma \ref{lemenv}.
\noindent We get the same estimate with the same method for $\sum_{y_0 \in \Ct_{m_n}^j} \tpa_{y_0}( \A_n^-)$.  

\noindent \textit{Proof of \ref{2.6b}} First notice that $\sum_{j=1}^{n_d}\Rf_{m_n}^j \leq 1$ and $n_d$ is bounded, with $\cup_{j=1}^{n_d}\Cm=\Ct_{m_n}$ therefore to get \ref{2.6b} from \ref{2.6}, we only have to check that:
\begin{eqnarray*}
\lim_{n \rightarrow +Ê\infty }\p^V_{0}\left(\left|\frac{\lo(\Cm,n)}{n}-\Rf_{m_n}^j\right|\geq \epsilon \Rf_{m_n}^j \right)=0. 
\end{eqnarray*}
Only the very begining of the computations differs from what we did above because the walk starts from $0$, by using \ref{2.2} we easily get that
\begin{eqnarray*}
& & \p^V_{0}\left(\left|\frac{\lo(\Cm,n)}{n}-\Rf_{m_n}^j\right|\geq \epsilon  \Rf_{m_n}^j \right) \\ & \leq & \sum_{z\in \Cm} \left(\p^V_{z}\left(\left|\frac{\lo(\Cm,n)}{n}-\Rf_{m_n}^j\right|\geq \epsilon  \Rf_{m_n}^j \right)+\p^V_{z}\left(\left|\frac{\lo(\Cm,n')}{n}-\Rf_{m_n}^j\right|\geq \epsilon  \Rf_{m_n}^j \right) \right)+o(1) . 
\end{eqnarray*}
with $n'=n(1-1/n^{\epsilon_n})$, and $\lim_{n \rightarrow + \infty }o(1)=0$. Then the computations remain the same as above. \\
\noindent \textit{Proof of \ref{2.8}}, we use the same method as above and we get the following inequality, for all $1 \leq j \leq n_d$ and $l$ such that $|l|\leq k$:
\begin{eqnarray*}
\lim_{n \rightarrow + \infty} \sum_{y_0\in \Ct_{m_n}^1}\p^V_{y_0}\left(\left|\frac{\lo(\Ce_{m_n+l}^j,T_{\Ct_{m_n}^1,n \Rf_{m_n}^1})}{n}-\Rf_{m_n+l}^j\right|\geq \epsilon \Rf_{m_n+l}^j \right)=0.
\end{eqnarray*}
This leads easily to the result by considering \ref{2.6} and the above discussion about \ref{2.6b}.
\end{Pre}

\subsection{Proof of Theorem \ref{th1.3} (annealed results)}
In this part we, prove a result in law for the random environment, we introduce the potential conditioned to remain positive $\bar{S}$, the constants $p_i$, this leads to our second theorem.

\noindent The main result of this section is the following

\begin{Pro} \label{cvloi2} For any $k \in \N$,
\begin{eqnarray*}
\left(\lo(n,\Ce_{m_n+l}),\ -k \leq l \leq k)\right) \cvloi (\bar{\Pi}(l),\ -k\leq l \leq k), 
\end{eqnarray*}
recall that $\bar{\Pi}$ is defined in \ref{1.14}.
\end{Pro}
\noindent Considering what we did in the previous paragraph, we only need to prove the following
\begin{Lem} \label{2.6c} Let $k \in \N^*$, there exists a sequence $(p_j,1 \leq j \leq n_d )$ of non negative terms satisfying $\sum_{j=1}^{n_d}p_j=1$ such that
\begin{eqnarray*}
\left(\sum_{j=1}^{n_d}\Rf_{m_n+i}^j, -k \leq i \leq k \right) \cvl \left(\frac{\sum_{j=1}^{n_d}p_j\bar{\pi}^{j}_i}{{\sum_{l=-\infty}^{+ \infty} \sum_{j=1}^{n_d}p_j\bar{\pi}^{j}_l }},\ -k \leq i \leq k \right), 
\end{eqnarray*}
recall that $\bar{\pi}_.^.$ is defined Section \ref{parSp}, and $\cvl$ is the convergence in law under $P_e$.
\end{Lem}

\begin{Pre}
First we prove that for all $k \in \N^*$, $\epsilon>0$, there exists $\gamma>0$ such that
\begin{eqnarray}
\lim_{n \rightarrow + \infty } P_e\left(\bigcap_{-k \leq l \leq k} \left\{ \left|  \sum_{j=1}^{n_d}\Rf_{m_n+l}^j - \frac{\sum_{j=1}^{n_d}F^j(m_n+l,m_n)}{\sum_{i=-(\log n)^{2-\epsilon} }^{(\log n)^{2-\epsilon}}{\sum_{j=1}^{n_d}F^j(m_n+i,m_n)}} \right| > \frac{\ct}{ (\log n)^{\gamma}} \right\} \right)=0,  \label{2.18}
\end{eqnarray}
where 
\begin{eqnarray*}
F^j(r,m_n):=\frac{|\Ct_{r}^j|}{|\Ct_{m_n}|}\ttpi^j_{r},\ \ttpi^j_{r}=\pi_r^j\exp(S_{m_n}). 
\end{eqnarray*}
Recalling the definition of $\Rf_{m_n+l}^j$,
\begin{eqnarray*}
\Rf_{m_n+l}^j= \frac{|\Ct_{m_n+l}^j|\pi_{m_n+l}^j}{ \sum_{x \in \B_{M_n}} \pi(x)} \equiv  \frac{|\Ct_{m_n+l}^j|/|\Ct_{m_n}| \ttpi_{m_n+l}^j }{1/|\Ct_{m_n}| \sum_{x \in \B_{M_n}} \pi(x)\exp(S_{m_n})} 
\end{eqnarray*}
we split $\sum_{x \in \B_{M_n}} \pi(x)\exp(-S_{m_n})$ into two parts
\begin{eqnarray}
\sum_{x \in \B_{M_n}} \pi(x)= \sum_{i=-(\log n)^{2-\epsilon}}^{(\log n)^{2-\epsilon}} \sum_{z \in \Ce_{m_n+i}} {\tpi(z)}\exp(S_{m_n}) + \sum_{|i|\geq (\log n)^{2-\epsilon}}\sum_{z \in \Ce_{m_n+i}\cap \B_{M_n}} {\tpi(z)}\exp(S_{m_n}). \nonumber
\end{eqnarray}
 For the second sum, with the help of Remark \ref{rem2.3}, we have $P_e.a.s.$
\begin{eqnarray*}
\sum_{z \in \Ce_{m_n+i}\cap \B_{M_n}} {\tpi(z)}\exp(S_{m_n}) \leq \ct |\Ct_{M_n}| \exp(-(S_{m_n+i}-S_{m_n})),
\end{eqnarray*} 
then thanks to Remark \ref{1.1} and statement 2. of Lemma \ref{lemenv}, with a $P_e$ probability converging to one
\begin{eqnarray*}
& & \frac{1}{|\Ct_{m_n}|}\sum_{|i|\geq (\log n)^{2-\epsilon}} \sum_{z \in \Ce_{m_n+i}\cap \B_{M_n}} {\tpi(z)}\exp(S_{m_n}) \\ &\leq&  \ct (\log n)^{\epsilon_0(d-1)} \sum_{|i|\geq (\log n)^{2-\epsilon}, -m_n \leq i \leq M_n-m_n }  \exp(-(S_{m_n+i}-S_{m_n})),
\end{eqnarray*} 
with $\epsilon_0>0$ to be chosen. Moreover we know (see for example the Appendix of \cite{Pierre3}) that for all $r$ 
\begin{eqnarray}
\lim_{n \rightarrow \infty }P_1\left(\sum_{ |i|\geq r, -m_n \leq i \leq M_n-m_n  }  e^{-(S_{m_n+i}-S_{m_n})} > \epsilon_0' \right) \leq \frac{\ct}{\epsilon_0' \sqrt{r}},  \label{lemP}
\end{eqnarray}
with $\epsilon_0'>0$. Assembling what we did above leads to: choosing $\epsilon_0'=(\log n)^{-1/4}$ there exists a $\gamma>0$ such that 
\begin{eqnarray*}
\lim_{n \rightarrow + \infty}P_e \left(\frac{1}{|\Ct_{m_n}|} \left| \sum_{x \in \B_{M_n}} \pi(x) - \sum_{i=-(\log n)^{2-\epsilon}}^{(\log n)^{2-\epsilon}} \sum_{z \in \Ce_{m_n+i}} {\tpi(z)}\exp(S_{m_n}) \right| \geq (\log n)^{-\gamma}  \right)=0.  
\end{eqnarray*}
Notice that the above normalized double sum is $P_e.a.s.$ larger than a  strictly positive constant and it can be re-written in the following way,
\begin{eqnarray*}
\frac{1}{|\Ct_{m_n}|} \sum_{i=-(\log n)^{2-\epsilon}}^{(\log n)^{2-\epsilon}} \sum_{z \in \Ce_{m_n+i}} \tpi(z)= \frac{1}{|\Ct_{m_n}|} \sum_{i=-(\log n)^{2-\epsilon}}^{(\log n)^{2-\epsilon}} \sum_{j=1}^{n_d} |\Ce^j_{m_n+i}| \tpi^j_{m_n+i}
 \end{eqnarray*}
moreover we can check that the numerator of $\Rf_{m_n+l}^j$: $|\Ct_{m_n+l}^j|/|\Ct_{m_n}| \pi_{m_n+l}^j \exp(S_{m_n})$ is bounded $P_e.a.s.$ so we get \ref{2.18}.


\begin{Lem}  There exists a sequence of non-negative numbers $(p_j, 1 \leq j \leq n_d)$,  with $\sum_{j=1}^{n_d} p_j=1$ such that  
\begin{eqnarray*}
\lim_{n \rightarrow + \infty } P_e\left( \bigcap_{r=- (\log n)^{2-\epsilon}}^{(\log n)^{2-\epsilon}} \bigcap_{j=1}^{n_d} \left\{ \left| |\Ct_{m_n+r}^j|/|\Ct_{m_n}|-p_j\right| \leq g(n) \right\} \right)=1,
\end{eqnarray*}
where $g$ is a positive decreasing function with $\lim_{n \rightarrow + \infty}g(n)=0$.
 \end{Lem}
 
 \begin{Pre}
 First, by independence of the $\delta's$ , it is easy to show that
 \begin{eqnarray*}
 \lim_{n \rightarrow + \infty } P_e\left( \bigcap_{r=- (\log n)^{2-\epsilon}}^{(\log n)^{2-\epsilon}} \left\{ \left| |\Ct_{m_n+r}|/|\Ct_{m_n}|-1\right| \leq g_1(n) \right\} \right)=1,
 \end{eqnarray*}
 where $g_1=(\log n)^{-\gamma_0}$, with some $\gamma_0>0$. We recall that $n_d$ is bounded and that $\sum_{j=1}^{n_d} |\Ct_{m_n+r}^j|/|\Ct_{m_n+r}|=1$, moreover by Kolmogorov's zero-one law, for all $j$ and $r$, $\lim_{n \rightarrow \infty}|\Ct_{m_n+r}^j|/|\Ct_{m_n+r}|$ exists $P_2.a.s$ and we call this limit $p_j$. This finish the proof.
 \end{Pre} \\
The above Lemma yields
\begin{eqnarray}
\lim_{n \rightarrow + \infty }P_e\left( \left| \frac{1}{|\Ct_{m_n}|}\sum_{i=-(\log n)^{2-\epsilon}}^{(\log n)^{2-\epsilon}} \sum_{j=1}^{n_d} |\Ce^j_{m_n+i}| \ttpi^j_{m_n+i}-\tilde{\sum}_{(\log n)^{2-\epsilon}}\right| \geq g(n) \tilde{\sum}_{(\log n)^{2-\epsilon}}\right)=0, \label{2.19b}
\end{eqnarray}
where 
\begin{eqnarray}
\tilde{\sum}_l:=\sum_{i=-l}^{l}   \tilde{F}_n(i),\textrm{ and }  \tilde{F}_n(r) = \sum_{j=1}^{n_d} p_j \ttpi^j_{m_n+r}. \label{2.19ab}
\end{eqnarray}
So \ref{2.18} and \ref{2.19b} leads to, for all $\epsilon_0''>0$
\begin{eqnarray*}
\lim_{n \rightarrow + \infty}P_e\left(\bigcap_{-k \leq l \leq k} \left\{ \left|\sum_{j=1}^{n_d}\Rf_{m_n+l}^j - \tilde{F}_n(l)/\tilde{\sum}_{(\log n)^{2-\epsilon}}\right| \geq \epsilon_0'' \right\} \right)=0. 
\end{eqnarray*}
\\
We are now moving to the convergence in law, we need to prove the following, let $\alpha\geq 0$
\begin{eqnarray}
\lim_{n \rightarrow + \infty} P_1\left(\tilde{F}_n(l)/\tilde{\sum}_{(\log n)^{2- \epsilon}} \leq \alpha \right) = P_1\left(\bar{F}(l)/\bar{\sum}_{\infty} \leq \alpha \right). \label{endloi}
\end{eqnarray}
\textit{First step} Let us denote $A_{n,n_1,\epsilon_1}=\left\{\left| \tF_n(l)/\tS_{n_1}-\tF_n(l)/\tSn \right| \leq \epsilon_1 \right\}$, with $\epsilon_1>0$. We show that 
\begin{eqnarray}
\DL P_1(A_{n,n_1,\epsilon_1})=1. \label{step1}
\end{eqnarray}
Thanks to hypothesis \ref{hyp3}, we only have to check that $\DL P_1\left(\left|1/\tS_{n_1}-1/\tSn \right| > \epsilon_1 \right)$ $=0$, keeping the same $\epsilon_1$ even if it changes a little bit. Assume $n_1 < (\log n)^{2- \epsilon}$, a little of computations yields
\begin{eqnarray*}
P_1\left(\left|1/\tS_{n_1}-1/\tSn \right| > \epsilon_1 \right) \leq P_1\left(\tS_{n_1, (\log n)^{2-\epsilon}}> \epsilon_1 (\tS_{n_1})^2 \right),
\end{eqnarray*}
then thanks to hypothesis \ref{hyp3} again, $\tS_{n_1} \geq \ct  >0$, statement 2 of Lemma \ref{lemenv} and \ref{lemP} we get 
\begin{eqnarray*}
P_1\left(\tS_{n_1, (\log n)^{2-\epsilon}}> \epsilon_1 (\tS_{n_1})^2 \right)\leq \frac{1}{\ct \epsilon_1}\frac{1}{\sqrt{n_1}}.
\end{eqnarray*}
To finish we take the limit for $n$, and finally for $n_1$ so we get \ref{step1}. \\
\textit{Second step} By using \ref{step1}, we easily get that
\begin{eqnarray*}
\DL P_1\left(\frac{\tilde{F}_n(l)}{\tS_{n_1}} \leq \alpha-\epsilon_1 \right) \leq \lim_{n \rightarrow + \infty} P_1\left(\frac{\tilde{F}_n(l)}{\tilde{\sum}_{(\log n)^{2- \epsilon}}}  \leq \alpha\right) \leq \DL P_1\left(\frac{\tilde{F}_n(l)}{\tS_{n_1}} \leq \alpha+\epsilon_1 \right),
\end{eqnarray*}
moreover from \cite{Golosov} Lemma 4 we know that the finite distribution of $(S_{m_n+i}-S_{m_n},i \in \Z)$ converge to those of $(\bar{S}_l,l)$ defined Section \ref{parSp}, that is to say
\begin{eqnarray*} 
\left( \frac{\tilde{F}_n(l)}{\tS_{n_1}}, -k \leq l \leq k \right) \rightarrow \left( \frac{ \bar{F}(l)}{\bS_{n_1}}, -k \leq l \leq k \right) , \ \textrm{in $P_1$ law} 
\end{eqnarray*}
where $\bar{F}$ (resp. $\bar{\sum}$) is given by \ref{2.19ab}, replacing $\tF_n$ by $\bar{F}$ (resp. $\tS$ by $\bar{\sum}$). The above inequality becomes 
\begin{eqnarray}
\lim_{n_1 \rightarrow + \infty} P_1\left(\frac{\bF(l)}{\bS_{n_1}} \leq \alpha-\epsilon_1 \right) \leq \lim_{n \rightarrow + \infty} P_1\left(\frac{\tilde{F}_n(l)}{\tilde{\sum}_{(\log n)^{2- \epsilon}}}  \leq \alpha\right) \leq\lim_{n_1 \rightarrow + \infty} P_1\left(\frac{\bar{F}(l)}{\bS_{n_1}} \leq \alpha+\epsilon_1 \right). \label{step2}
\end{eqnarray}
We are almost done: let $\gamma_1>0$, denote $B_{n_1}:= \left\{ (1-\gamma_1) \bar{\sum}_{+\infty} \leq  \bar{\sum}_{n_1} \leq \bar{\sum}_{+\infty} \right \}$, for all $\mu \geq 0$, we have
\begin{eqnarray*}
P_1\left(\bar{F}(l)/(\bar{\sum}_{+\infty}(1-\gamma_1)) \leq \mu,\ B_{n_1} \right) \leq P_1\left(\bar{F}(l)/\bar{\sum}_{n_1} \leq \mu,B_{n_1} \right)  \leq P_1\left(\bar{F}(l)/\bar{\sum}_{+\infty} \leq \mu  \right)
 \end{eqnarray*}
  thanks to Bertoin \cite{Bertoin}, $\lim_{n_1 \rightarrow + \infty }P\left( B_{n_1} \right)=1 $, so 
\begin{eqnarray*}
P_1\left(\bar{F}(l)/(\bar{\sum}_{+\infty}(1-\gamma_1)) \leq \mu \right) \leq \lim_{n_1 \rightarrow + \infty}P_1\left(\bar{F}(l)/\bar{\sum}_{n_1} \leq \mu \right)  \leq P_1\left(\bar{F}(l)/\bar{\sum}_{+\infty} \leq \mu  \right)  
  \end{eqnarray*} 
 letting $\gamma_1$ goes to zero, inserting the result in \ref{step2} and letting $\epsilon_1$ goes to zero we get \ref{endloi}
 \end{Pre}
\\
Proposition \ref{cvloi2} is a consequence of \ref{2.8} and Lemma \ref{2.6c}.

\subsection{Proof of Theorem \ref{th1.1}}

The proof can be deduced from Proposition \ref{cvloi2} and the same method exposed in \cite{GanPerShi} pages 6 and 7.


\section{Appendix}

\subsection{Basic formula for reversible random walks}


We recall and shortly prove basic results on the moments of the local time of nearest neighborhood reversible random walks. Assume that $\A$ and $\A'\subset \Z^d$ are two sets where the reversible measure $\pi$ is constant, also we recall that the different hitting times $T$ are defined at the begining of Section \ref{th1.3}, we have 

\begin{Lem} \label{Lemmoy} Let $x\in \Z^d$, $y\in \A$, $v \in \A'$ then
\begin{eqnarray} 
 \sum_{z \in \A}\Ea_z(\lo(x,T^+_{\A}))=\frac{\pi(x)}{\pi(y)}, \label{ap1} \\
 \sum_{z \in \A}\Ea_z(\lo(\A',T^+_{\A}))=|\A'|\frac{\pi(v)}{\pi(y)}, \label{ap2}
 \end{eqnarray}
 where $|\A'|$ is the size (cardinal) of $\A'$. 
\end{Lem} 
\noindent For completness, a few words of the proof  \\
\begin{Pre} For \ref{ap1} It is easy to get that 
\begin{eqnarray*}
 \Ea_z(\lo(x,T^+_{\A}))=\frac{\pa_z(T_x<T^+_{\A})}{\pa_x(T_x^+>T^+_{\A})}, 
 \end{eqnarray*}
moreover the chain is reversible therefore 
\begin{eqnarray*}
 \pa_x(T_x^+>T^+_{\A})=\sum_{u \in \A}\frac{\pi(u)}{\pi(x)}\pa_u(T_x<T^+_{\A}), 
 \end{eqnarray*}
 so we get the lemma by definition of $\A$. \ref{ap2} follows immediatly from \ref{ap1}. 
\end{Pre} 

\begin{Lem} \label{Lemvar}  Let $x \in \A$ and $z \notin \A$ then
\begin{eqnarray} 
\sum_{y \in \A}\Ea_{y}\left[\left(\lo(z,T^+_{\A})-\frac{ \pi(z)}{|\A|\pi(x)}\right)^2\right] \leq  \frac{2\pi(z)}{\pi(x)} \frac{1}{\pa_z(T_z^+>T^+_{\A})}.
\end{eqnarray}
\end{Lem} 
\begin{Pre} 
With the same idea as the proof above we easily get 
\begin{eqnarray*} 
\sum_{y \in \A}\Ea_{y}\left[\left(\lo(z,T^+_{\A})\right)^2\right] = \frac{\pi(z)}{\pi(x)} \frac{1}{\pa_z(T_z^+>T^+_{\A})}+ \frac{\pi(z)}{\pi(x)},
\end{eqnarray*}
and using Lemma  \ref{Lemmoy} for the other term we finish the proof. 
\end{Pre}

\noindent \\ 
For the study of the excursion of the walk we also need the following elementary results: 
\begin{Lem} Let $x\in \Z^d$, $y\in \A$ and $v\in \A'$, and $l \in \N^*$ then
\begin{eqnarray} 
& & \sum_{z \in \A}\Ea_z(\lo(x,T^+_{\A,l}))=l \frac{\pi(x)}{\pi(y)}, \label{ap3}\\ 
 & &  \sum_{z \in \A}\Ea_z(\lo(\A',T^+_{\A,l}))=l|\A'|\frac{\pi(v)}{\pi(y)} \label{ap4}.
 \end{eqnarray}
\end{Lem} 

\begin{Pre}
If $\A$ is a singleton then it is trivial because the sequence $(T^+_{\A,l}-T^+_{\A,l-1},l)$ is i.d. When $\A$ is not a singleton we can get the result recursively, indeed, we easily have:
\begin{eqnarray*}
\Ea_z(\lo(x,T^+_{\A,l}))& = & \sum_{u \in \A} \Ea_z(\lo(x,T^+_{\A}) \un_{X_{T_{\A}}=u})+\sum_{u \in \A} \Ea_u(\lo(x,T^+_{\A,l-1})) \p_z({X_{T_{\A}}=u}) \\
& \equiv & \Ea_z(\lo(x,T^+_{\A}) )+\sum_{u \in \A} \Ea_u(\lo(x,T^+_{\A,l-1})) \p_z({X_{T_{\A}}=u})
 \end{eqnarray*}
and thanks to the reversibility and the definition of $\A$, $\p_z({X_{T_{\A}}=u})=\p_u({X_{T_{\A}}=z})$, therefore 
\begin{eqnarray*}
\sum_{z \in \A}\Ea_z(\lo(x,T^+_{\A,l}))& = &  \sum_{z \in \A}\Ea_z(\lo(x,T^+_{\A}) )+\sum_{u \in \A} \Ea_u(\lo(x,T^+_{\A,l-1})) 
 \end{eqnarray*}
 so we get the Lemma using \ref{ap1}.
\end{Pre} 
 
\noindent \\ Let us denote  $e_x=\sum_{u \in \A}\Ea_u(\lo(x,T^+_{\A}))$.
 \begin{Lem} \label{lem3.3} Let $x\in \Z^d$, $y\in \A$ and $l \in \N^*$ then
\begin{eqnarray} 
 & & \sum_{z \in \A} \Ea_z\left[\left(\lo(x,T^+_{\A,l})-\frac{l}{|\A|}e_x\right)^2\right] \nonumber \\ 
 & = &  l \sum_{z \in \A} \Ea_z\left[\left(\lo(x,T^+_{\A})-\frac{1}{|\A|}e_x\right)^2\right]  \nonumber\\ 
 &+&2 \sum_{i=1}^{l-1} (l-i)  \sum_{v_i\in \A} \sum_{v \in \A} E_{v_i}\left[ \left(\lo(x,T^+_{\A})-\frac{1}{|\A|}e_x\right)\right]P_v(X_{T^+_{\A,i}}=v_i) E_v\left[\left(\lo(x,T^+_{\A})-\frac{1}{|\A|}e_x\right)\right]
. \nonumber \end{eqnarray}
\end{Lem} 

 \begin{Pre} The Lemma can be proven recursively,  adding and substrating $\lo(x,T^+_{\A,l-1})$, we get that 
 \begin{eqnarray}
 \left(\lo(x,T^+_{\A,l})-\frac{l}{|\A|}e_x\right)^2 & = &
 \left(\lo(x,T^+_{\A,l-1})-\frac{l-1}{|\A|}e_x\right)^2 +\left(\lo(x,T^+_{\A,l})-\lo(x,T^+_{\A,l-1})-\frac{1}{|\A|}e_x\right)^2  \nonumber \\
 &+& 2   \left(\lo(x,T^+_{\A,l-1})-\frac{l-1}{|\A|}e_x\right) \left(\lo(x,T^+_{\A,l})-\lo(x,T^+_{\A,l-1})-\frac{1}{|\A|}e_x\right) \nonumber
 \end{eqnarray}
 let us denote $F(l)=\sum_{z \in \A}\Ea_z\left( \left(\lo(x,T^+_{\A,l})-\frac{l}{|\A|}e_x\right)^2\right)$ and let $F(0)=0$, using the above decomposition the reversibility of the Markov chain, the definition of $\A$ and the strong Markov property
 \begin{eqnarray*} 
F(l)&=& F(l-1)+ \sum_{z \in \A} \Ea_z\left[\left(\lo(x,T^+_{\A})-\frac{1}{|\A|}e_x\right)^2\right]  \nonumber \\
&+& 2 \sum_{v\in \A}\sum_{z\in \A} \Ea_z\left[ \left(\lo(x,T^+_{\A,l-1})-\frac{l-1}{|\A|}e_x\right)\un_{X_{T^+_{\A,l-1}}=v}\right]\Ea_v\left[\left(\lo(x,T^+_{\A})-\frac{1}{|\A|}e_x\right)\right].
\end{eqnarray*}
The reversibility gives 
\begin{eqnarray*} 
 \sum_{z\in \A} \Ea_z\left[ \left(\lo(x,T^+_{\A,l-1})-\frac{l-1}{|\A|}e_x\right)\un_{X_{T^+_{\A,l-1}}=v}\right]
 &=& \Ea_v\left[ \left(\lo(x,T^+_{\A,l-1})-\frac{l-1}{|\A|}e_x\right)\right]\end{eqnarray*} 
recursively on $l$ we also get that  
\begin{eqnarray*} 
 & & \sum_{v\in \A} \Ea_v\left[ \left(\lo(x,T^+_{\A,l-1})-\frac{l-1}{|\A|}e_x\right)\right]\Ea_v\left[\left(\lo(x,T^+_{\A})-\frac{1}{|\A|}e_x\right)\right] \nonumber \\
 &=&  \sum_{i=1}^{l-1} \sum_{v_i\in \A} \sum_{v \in \A} \Ea_{v_i}\left[ \left(\lo(x,T^+_{\A})-\frac{1}{|\A|}e_x\right)\right]\pa_v(X_{T^+_{\A,i}}=v_i) \Ea_v\left[\left(\lo(x,T^+_{\A})-\frac{1}{|\A|}e_x\right)\right]. 
\end{eqnarray*} 
Putting together what we did above gives the Lemma.
\end{Pre}  
\noindent \\


\begin{Lem} \label{A.14}For all $l \inÊ\N^*$ and $u,u_0 \in \A$, $|\pa_{u_0}\left[X_{T_{\A,l}}=u\right]-1/|\A|| \leq \left(1-1/|\A|\right)^l$. 
\end{Lem}

\begin{Pre}
We have a reversible Markov chain with finite space state and symmetric distribution, proof is basic, see \cite{RusPer} for example.
\end{Pre}

\subsection{The Dirchlet method}
The Dirchlet method (see for example \cite{Liggett}) allows us to get a part of the estimate we need in this paper, we recall that for a reversible Markov chain, we have the following elementary result :
Let $z \in \Z^d$, and $A \subset \Z^d$ such that $x \notin A$, $T_{\bar{A}}= \inf \{k>0: X_k \notin A   \}$, $T_z= \inf \{k>0: X_k=z   \}$.  Define also $\Ht_{\bar{A},z}=\{f: \Z^d \rightarrow [0,1]:\ f(z)=0 \textrm{ and }f(y)=1 \textrm{ for }y \notin A\}$, $f_{\bar{A}}(y)=\pa_y(T_{\bar{A}} \leq T_z)\un_{y \neq z}$, and $\Phi(f)= \sum_{y,z}\pi(y)p(y,z)(f(y)-f(z))^2$ then if $\pa_z(T_{\bar{A}}<\infty)=1$ then $2\pi(z)\pa_z(T_{\bar{A}} \leq T_z)=\Phi(f_A)=\min_{f \in \Ht_{\bar{A},z}}\Phi(f)$.
Applied in our context, this leads to
\begin{Lem} \label{Lem4.6} Let $k \in \N^*$, $l \in \N^*$, $h_k$ a function such that $h_k(y)=1$ for any $y \in \B_k$, and $h_k(x)=0$ for all $x \in \bar{\B}_k$ then 
\begin{eqnarray}
& &\Phi(h_k)=\sum_{y\in \Ct_{k+1}}\sum_{z\in \Ct_{k}}e^{-1/2V(y)-1/2V(z)}\un_{|y-z|=1}, 
\end{eqnarray}
let $z\in \B_k$ then 
\begin{eqnarray}
\pa_z(T_z^+ \geq T_{\bar{\Bt}_k}) \leq \frac{\Phi(h_k)}{2\pi(z)}. \label{4.17}
\end{eqnarray}
Let $(z_q, 0<q<m)$, a self avoiding path from $z_0=z$ to $z_m=y$ then 
\begin{eqnarray}
\pa_z(T_y<T_{\bar{\Bt}_k}) \geq 1- \left( \frac{ \Phi(h_k)  m
}{\inf_{q} \pi(z_{q-1},z_q)}  \right)^{1/2},
\end{eqnarray}
$(z_q', 0<q<m')$, a self avoiding path from $z_0'=z \in \Bt_{k}$ to $z_m'=w \in {\Ct}_{k+1}$ and belonging to $\Bt_{k} \smallsetminus \Bt_{\bar{z}-1}$  then
\begin{eqnarray}
\pa_z(T_z^+ \geq T_{\Ct_{k+1}})  \equiv \pa_z(T_z^+ \geq T_{\bar{\Bt}_k}) \geq  \frac{  \inf_{q} \pi(z_{q-1}',z_q') }{2m' \pi(z)},  \label{4.19}
\end{eqnarray}
$(z_q'', 0<q<m'')$, a self avoiding path from $z_0''=z \in \Bt_{k+l} \smallsetminus \Bt_{k}$ to $z_m''=w \in {\Ct}_{k}$ and belonging to $\Bt_{\bar{z}}  \smallsetminus \Bt_{k}$ then
\begin{eqnarray}
\pa_z(T_z^+ \geq T_{\Ct_{k}}) \equiv \pa_z(T_z^+ \geq T_{\Bt_k}) \geq  \frac{  \inf_{q} \pi(z_{q-1}'',z_q'') }{2m'' \pi(z)} . \label{4.19b}
\end{eqnarray}
\end{Lem}

\begin{Pre}
The proof can be found, for example, between pages 98 to 101 of \cite{Durett}, for completness we recall the main steps. The first equality and inequality is a direct consequence of what we recalled  above, notice that $h_k \in \Ht_{\bar{\Bt}_k,z}$ and $\Phi(h_k) \geq 2\pi(z)\pa_z(T_{\bar{\Bt}_k} \leq T_z)$. \\
Let us denote $g_{k}(y)=1-\pa_y(T_{\bar{\Bt}_k} \leq T_z)\un_{y \neq z}$. For the third inequality, we have
\begin{eqnarray*}
\Phi(g_k) & \geq & \sum_{q=1}^m \pi(z_{q-1},z_q)(g_k(z_q)-g_k(z_{q-1}))^2 \\
 & \geq & \inf_{1 \leq q \leq m }  \left(\pi(z_{q-1},z_q) \right) \sum_{q=1}^m (g_k(z_q)-g_k(z_{q-1}))^2  \\
 & \geq &  \frac{1}{m}\inf_{1 \leq q \leq m }  \left(\pi(z_{q-1},z_q) \right)  (g_k( z)-g_k(y))^2
\end{eqnarray*}
where the last inequality comes from Cauchy-Schwarz inequality, to finish notice that by definition $\Phi(g_k) \leq \Phi(h_k)$ and $g_k(z)=1$. Finally to prove \ref{4.19}, by following what we did above we get
\begin{eqnarray*}
\Phi(g_k) 
 & \geq &  \frac{1}{m'}\inf_{1 \leq l \leq m' }  \left(\pi(z_{l-1}',z_l') \right) 
\end{eqnarray*}
because $g_k(z)=1$ and $g_k(w)=0$ for all $w \in \bar{\Bt}_k$, the proof is complete because $\Phi(g_k)=2\pi(z)\pa_z(T_z^+>T_{\bar{\Bt}_k}) $. Proof of \ref{4.19b} is the same as above.
\end{Pre}

\bibliographystyle{alpha}

 \bibliography{thbiblio}

\vspace{1cm} \noindent
\begin{tabular}{l}
Laboratoire MAPMO - C.N.R.S. UMR 6628 - F\'ed\'eration Denis-Poisson  \\
Universit\'e d'Orl\'eans, UFR Sciences \\
B\^atiment de math\'ematiques - Route de Chartres \\
B.P. 6759 - 45067 Orl\'eans cedex 2 \\
FRANCE 
\end{tabular}

\end{document}